\documentclass[12pt,draft]{article}

\usepackage{mathtools}
\usepackage{amssymb}
\usepackage{amsthm}
\usepackage{xcolor}
\usepackage{tikz}
\usepackage{fullpage}
\usepackage{url}

\DeclareSymbolFont{bbold}{U}{bbold}{m}{n}
\DeclareSymbolFontAlphabet{\mathbbold}{bbold}

\newcommand{\N}{\mathbb{N}}

\newcommand{\R}{\mathbb{R}}

\newcommand{\1}{\mathbbold{1}}

\renewcommand{\epsilon}{\varepsilon}



\newcommand{\fin}{\rm fin}

\providecommand{\form}{\tau}
\providecommand{\scpr}[2]{\left( #1 \,\middle|\, #2 \right)}
\providecommand{\dupa}[2]{\left\langle #1 \,,\, #2 \right\rangle}
\renewcommand{\sp}{\scpr}
\newcommand{\from}{\colon}

\let\phi\varphi

\let\leq\leqslant

\let\geq\geqslant

\makeatletter
\def\@row#1,{#1\@ifnextchar;{\@gobble}{&\@row}}
\def\@matrix{%
    \expandafter\@row\my@arg,;%
    \@ifnextchar({\\ \get@in@paren{\@matrix}}{\after@matrix}%
    }
\def\matrixtype#1#2#3{%
    \ifmmode\def\after@matrix{\end{#2}\right#3}%
    \else\def\after@matrix{\end{#2}\right#3$}$\fi
    \left#1\begin{#2}\get@in@paren{\@matrix}%
    }
\def\@column#1,{#1\@ifnextchar;{\@gobble}{\\ \@column}}
\newcommand\vect{}
\def\svect(#1){\left(\begin{smallmatrix}\@column#1,;\end{smallmatrix}\right)}
\def\vect{\get@in@paren{\@vect}}
\def\@vect{\left(\begin{matrix}\expandafter\@column\my@arg,;\end{matrix}\right)}
\def\get@in@paren#1({\def\my@arg{}\def\my@rest{}\def\after@get{#1}\get@arg}
\let\e@a\expandafter
\def\get@arg#1){\e@a\kl@test\my@rest#1(;}
\def\kl@test#1(#2;{\e@a\def\e@a\my@arg\e@a{\my@arg#1}%
                   \ifx:#2:\let\my@exec\after@get
                   \else\let\my@exec\get@arg
                        \e@a\def\e@a\my@arg\e@a{\my@arg(}%
                        \def@rest#2;%
                   \fi\my@exec}
\def\def@rest#1(;{\def\my@rest{#1\kl@zu}}
\def\kl@zu{)}

\makeatletter
\newcommand\MyPairedDelimiter{%
  \@ifstar{\My@Paired@Delimiter{{}}}
          {\My@Paired@Delimiter{}}%
}
\newcommand\My@Paired@Delimiter[4]{%
  \newcommand#2{%
    \@ifstar{\start@PD{#1}{\delimitershortfall=-1sp}{#3}{#4}}
            {\start@PD{#1}{}{#3}{#4}}%
  }%
}
\newcommand\start@PD[5]{%
  #1\mathopen{\mathpalette\put@delim@helper{\put@delim{#2}{#3}{.}{#5}}}%
  #5%
  \mathclose{\mathpalette\put@delim@helper{\put@delim{#2}{.}{#4}{#5}}}%
}
\newcommand\put@delim@helper[2]{%
  \hbox{$\m@th\nulldelimiterspace=0pt #2#1$}%
}
\newcommand\put@delim[5]{%
  \setbox\z@\hbox{$\m@th#5{#4}$}%
  \setbox\tw@\null
  \ht\tw@\ht\z@ \dp\tw@\dp\z@
  #1#5%
  \left#2\box\tw@\right#3%
}

\makeatother
\MyPairedDelimiter*{\abs}{\lvert}{\rvert}
\MyPairedDelimiter*{\norm}{\lVert}{\rVert}
\MyPairedDelimiter{\set}{\{}{\}}

 \theoremstyle{plain} 
  \newtheorem{theorem}{Theorem}[section]
  \newtheorem{corollary}[theorem]{Corollary}
  \newtheorem{lemma}[theorem]{Lemma}
  
  \theoremstyle{definition}
  
  \newtheorem{case}[theorem]{Case}
  
  \newtheorem*{definition}{Definition}
  \newtheorem{remark}[theorem]{Remark}

\usepackage{enumitem}

\setenumerate[1]{nolistsep} 
\setenumerate[2]{nolistsep} 

\begin{document}

\medmuskip=4mu plus 2mu minus 3mu
\thickmuskip=5mu plus 3mu minus 1mu
\belowdisplayshortskip=9pt plus 3pt minus 5pt

\title{Perturbations of positive semigroups on $L_p$-spaces}

\author{Christian Seifert \and Marcus Waurick}


\date{}

\maketitle

\begin{abstract}
 We give a characterization of a variation of constants type estimate relating two positive semigroups on (possibly different) $L_p$-spaces to one another in terms of corresponding estimates for the respective generators and of estimates for the respective resolvents.
 The results have applications to kernel estimates for semigroups induced by accretive and non-local forms on $\sigma$-finite measure spaces. 
\end{abstract}



\section{Introduction}

In this note, we shall further investigate a result recently obtained in \cite{SeifertWingert2014}.
In that article, the authors gave a semigroup-theoretic proof of a perturbation result obtained in \cite[Lemma 3.1]{BarlowGrigoryanKumagai2009}, 
which was subsequently used to derive heat kernel estimates for non-local Dirchlet forms, see 
\cite{BarlowBassChenKassmann2009,ChenKumagai2008,Foondun2009,GrigoryanHu2008}.
The aim of the present paper is to state a more general version of the perturbation 
result in \cite{SeifertWingert2014}. As a corollary, we recover the main result of \cite{SeifertWingert2014} as a 
special case. As it turns out, the generalization obtained here leads to a  
characterization of a weak formulation of a variations of constants type 
estimate for positive semigroups.

To be more precise, let $(\Omega,\mathcal{F},m)$ be a localizable measure space,
$1\leq p,q<\infty$ and $S$ and $T$ be positive $C_0$-semigroups on $L_2(m)$ (over $\R$).
Assuming that $S$ also acts on $L_p(m)$ and $T$ also acts on $L_q(m)$ we will characterize the validity of the (scalar) estimate
\[\dupa{T(t) u}{v'}_{2,2} \leq \dupa{S(t) u}{v'}_{2,2} + C \int_0^t \dupa{T(t-s)u}{g}_{q,q'} \dupa{f}{S(s)'v'}_{p,p'} \,ds \quad (t\geq 0)\]
for suitable positive elements $u,f,g,v'$ by means of the generators of the two semigroups, and also by their resolvents. Here, 
\[\dupa{\varphi}{\psi}_{p,p'}:=\int_\Omega \varphi\psi\,dm\]
denotes the dual pairing of $\varphi\in L_p(m)$ and $\psi\in L_{p'}(m) = L_p(m)'$ with $p' = \frac{p}{p-1}$ the dual exponent.
We shall also provide a strong formulation of this characterization, which can be viewed as a perturbation result for positive semigroups on $L_p$-spaces.

The results obtained have applications to kernel estimates for positive semigroups acting on $L_p$-spaces,
and can be used to derive heat kernel estimates (and in some cases also the existence of a kernel) for semigroups induced by Dirichlet forms perturbed by non-local parts. 
This leads to a further generalization of the heat kernel estimates obtained in \cite{BarlowGrigoryanKumagai2009}. 
In fact, for some measure space $(\Omega, \mathcal{F}, m)$, consider a closed accretive form $\tau_0$ on $L_2(m)$ such that the functions of finite support are dense in $D(\form_0)$ (see below for precise definitions).
Let $\form_0$ induce a positive and $L_\infty$-contractive (i.e.\ submarkovian) semigroup $S$. Perturbing $\tau_0$ additively by $\tau_j$ given by
 \[
    \tau_j(u)\coloneqq \int_{\Omega\times\Omega} (u(x)-u(y))^2j(x,y)dm^2(x,y)
 \]
for some measurable $j\colon \Omega\times \Omega \to [0,\infty)$ satisfying certain integrability assumptions,
we obtain kernel estimates for the semigroup induced by $\tau\coloneqq \tau_0 + \tau_j$ in terms of the kernel of $S$
for $j$ belonging to a sum of $L_p$-spaces in a sense to be specified below.
Established results, see \cite{BarlowGrigoryanKumagai2009,SeifertWingert2014}, work for bounded $j$ and regular Dirichlet forms $\tau_0$.

The paper is organised as follows.
In Section \ref{sec:conv}, we provide a technical lemma (Lemma \ref{lem:convolutions}) to be used in the main theorem of this article (Theorem \ref{thm:main}).
As it turns out, this lemma can be viewed to describe a behavior of certain convolutions and, thus, may be of independent interest.
Section \ref{sec:main_pos_semi} introduces the framework and contains the statement of our main theorem, and we also state some variants of it.
In Section \ref{sec:kernel_estimate} we focus on kernel estimates assuming that both semigroups have kernels. There we also show existence of a kernel for $T$ in the case $q=1$.
The concluding Section \ref{sec:application_forms} deals with an application to heat kernel estimates for non-local forms. 

\section{An auxiliary result}\label{sec:conv}

In this section we provide a key lemma needed in the proof of our main theorem.

\begin{lemma}
\label{lem:convolutions}
  For $n\in\N$ let $\varphi_n,\psi_n\from[0,\infty)\times \N_0\to \R$ be 
  continuous in the first variable, $\varphi,\psi\from [0,\infty)\to\R$,  with $\varphi_n(\cdot,n)\to \varphi$ uniformly on compacts, $\psi_n(\cdot,n)\to \psi$ uniformly on compacts.
  Assume $\varphi_n(t,l) = \varphi_l(t\cdot\frac{l}{n},l)$ for all $n\in\N,l\in\N,t\geq0$, and similarly for $\psi_n$. Then
  \[\frac{1}{n-1}\sum_{l=1}^{n-1} \varphi_{n-1}(t,n-l)\psi_{n-1}(t,l)\to \int_0^1 \varphi(t(1-s))\psi(ts)\,ds = \int_0^t \varphi(t-s)\psi(s)\,ds \quad(t\geq0).\]
\end{lemma}

\begin{proof}
  For $n\geq 2$ define $\mu_n:=\frac{1}{n-1} \sum_{l=1}^{n-1}\delta_l$, $\delta_l$ the Dirac measure at $l$, and $\nu_n:= \mu_n(n\cdot)$. Then, for $\lambda$ denoting the Lebesgue measure, $\nu_n\to \lambda$ weakly on $(0,1)$.
  Let $t\geq 0$. We compute
  \begin{align*}
  \frac{1}{n-1}\sum_{l=1}^{n-1} \varphi_{n-1}(t,n-l)\psi_{n-1}(t,l) & = \int_0^{n} \varphi_{n-1}(t,n-l)\psi_{n-1}(t,l)\,d\mu_n(l) \\
   & = \int_0^1 \varphi_{n-1}(t,n(1-l))\psi_{n-1}(t,nl)\,d\nu_n(l) \\
   & = \int_0^1 \varphi_{n(1-l)}(t\tfrac{n}{n-1}(1-l),n(1-l)) \psi_{nl}(t\tfrac{n}{n-1}l,nl)\,d\nu_n(l).
  \end{align*}
  Let $\varepsilon>0$, $K:=[0,2t]$. Then there exist $0<a<b<1$ such that
  \[\sup_{n\in\N}\abs{\bigg(\int_{0}^a + \int_b^1\bigg)\varphi_{n(1-l)}(t\tfrac{n}{n-1}(1-l),n(1-l)) \psi_{nl}(t\tfrac{n}{n-1}l,nl)\,d\nu_n(l)} \leq \varepsilon.\]
  Thus, it suffices to show that
  \[\int_a^b \varphi_{n(1-l)}(t\tfrac{n}{n-1}(1-l),n(1-l)) \psi_{nl}(t\tfrac{n}{n-1}l,nl)\,d\nu_n(l) \to \int_a^b \varphi(t(1-s))\psi(ts)\,ds.\]
  There exists $N\in\N$ such that for all $n\geq N$ we have
  \[\norm{\varphi_n(\cdot,n)-\varphi}_{\infty,K},\norm{\psi_n(\cdot,n)-\psi}_{\infty,K}\leq \varepsilon,\]
  \[\abs{\int_0^1 \varphi(t(1-s)\psi(ts)\,d(\nu_n(s)-\lambda(s))} \leq \varepsilon,\]
  and
  \[\sup_{l\in[0,1]} \abs{\varphi(t\tfrac{n}{n-1}(1-l))\psi(t\tfrac{n}{n-1}l) - \varphi(t(1-l))\psi(tl)} \leq \varepsilon\]
  since $\varphi$ and $\psi$ are uniformly continuous on $K$.
  For $n\geq N\max\{(1-a)^{-1},(1-b)^{-1},a^{-1},b^{-1}\}\geq N$ we obtain
  \begin{align*}
    & \abs{\int_a^b \varphi_{n(1-l)}(t\tfrac{n}{n-1}(1-l),n(1-l)) \psi_{nl}(t\tfrac{n}{n-1}l,nl)\,d\nu_n(l)  - \int_a^b \varphi(t(1-s))\psi(ts)\,ds} \\
    & \leq \int_a^b \abs{ \varphi_{n(1-l)}(t\tfrac{n}{n-1}(1-l),n(1-l)) \psi_{nl}(t\tfrac{n}{n-1}l,nl) - \varphi(t\tfrac{n}{n-1}(1-l))\psi(t\tfrac{n}{n-1}l)}\,d\nu_n(l) \\
    & \quad + \int_a^b\abs{\varphi(t\tfrac{n}{n-1}(1-l))\psi(t\tfrac{n}{n-1}l) - \varphi(t(1-l))\psi(tl)}\, d\nu_n(l) \\
    & \quad + \abs{\int_a^b \varphi(t(1-s))\psi(ts)\,d(\nu_n(s)-\lambda(s))} \\
    & \leq \int_a^b \abs{ \varphi_{n(1-l)}(t\tfrac{n}{n-1}(1-l),n(1-l)) - \varphi(t\tfrac{n}{n-1}(1-l))}\abs{ \psi_{nl}(t\tfrac{n}{n-1}l,nl)} \, d\nu_n(l) \\
    & \quad + \int_a^b \abs{\varphi(t\tfrac{n}{n-1}(1-l))}\abs{\psi_{nl}(t\tfrac{n}{n-1}l,nl)-\psi(t\tfrac{n}{n-1}l)} \,d\nu_n(l) + \varepsilon + \varepsilon \\
    & \leq \bigl(\sup_{n\in\N, s\in K} \abs{\psi_n(s,n)} + \sup_{s\in K}\abs{\varphi(s)}\bigr)\varepsilon + 2\varepsilon.
  \end{align*}\qed
\end{proof}

\section{A perturbation result}\label{sec:main_pos_semi}

We start by recalling the definition of a localizable measure space.
\begin{definition}
  Let $(\Omega,\mathcal{F},m)$ be a measure space. Then $(\Omega,\mathcal{F},m)$ is called \emph{semifinite} if for all $A\in \mathcal{F}$ with $m(A) = \infty$ there exists $B\in\mathcal{F}$ such that $B\subseteq A$ and $0<m(B)<\infty$.
  Furthermore, we say that $(\Omega,\mathcal{F},m)$ is \emph{localizable} if it is semifinite and for all $\mathcal{E}\subseteq \mathcal{F}$ there exists $A\in\mathcal{F}$ such that
  \begin{enumerate}
    \item $m(E\setminus A) = 0$ for all $E\in\mathcal{E}$,
    \item If $B\in \mathcal{F}$ and $m(E\setminus B)=0$ for all $E\in\mathcal{E}$, then $m(A\setminus B) = 0$.
  \end{enumerate}
\end{definition}
Note that every $\sigma$-finite measure space is localizable, and that a measure space is localizable if and only if $L_1(m)'=L_\infty(m)$ (with the canonical embedding), see e.g.~\cite{Segal1951,Fremlin2001}.

Let $(\Omega,\mathcal{F},m)$ be a localizable measure space, $1\leq p,q<\infty$. Let $p'$ be the dual exponent defined by $\frac{1}{p}+\frac{1}{p'} = 1$ (where $\frac{1}{\infty}:=0$), and similarly for $q$.
For $f\in L_p(m)$, $g\in L_{p'}(m)$ ($=L_p(m)'$) we write
\[\dupa{f}{g}_{p,p'} = \int_\Omega f g\,dm\]
for the dual pairing.
All vector spaces appearing will be real-valued and over the field $\R$.

\begin{lemma}
\label{lem:conv_int}
  Let $S$ be a $C_0$-semigroup on $L_p(m)$, $T$ a $C_0$-semigroup on $L_q(m)$. Let $f\in L_p(m)$, $v'\in L_{p'}(m)$, $u\in L_q(m)$, $g'\in L_{q'}(m)$. Then
  \[\frac{1}{t}\int_0^t \dupa{T(t-s)u}{g'}_{q,q'} \dupa{f}{S(s)'v'}_{p,p'}\,ds \to \dupa{u}{g'}_{q,q'}\dupa{f}{v'}_{p,p'} \quad(t\to 0).\]
\end{lemma}

\begin{proof}
  We compute
  \begin{align*}
    & \frac{1}{t}\int_0^t \dupa{T(t-s)u}{g'}_{q,q'} \dupa{f}{S(s)'v'}_{p,p'}\,ds - \dupa{u}{g'}_{q,q'}\dupa{f}{v'}_{p,p'} \\
    & = \frac{1}{t}\int_0^t \dupa{T(t-s)u}{g'}_{q,q'} \dupa{f}{(S(s)'-I)v'}_{p,p'}\, ds\\
    & \quad + \frac{1}{t}\int_0^t \dupa{(T(t-s)-I)u}{g'}_{q,q'}\dupa{f}{v'}_{p,p'} \,ds.
  \end{align*}
  Since $T$ is locally uniformly bounded and $S(\cdot)'$ is weakly$^*$ continuous, the first term on the right-hand side tends to zero as $t\to 0$. Since $T$ is strongly continuous the second term on the right-hand side tends also to zero.
  Thus, the assertion follows.\qed
\end{proof}

Let $T$ be a $C_0$-semigroup on $L_p(m)$ with generator $A$. Assume there exist $M\geq 1$, $\omega\in\R$ such that
\[\norm{T(t)u}_q\leq Me^{\omega t}\norm{u}_q \quad(u\in L_p\cap L_q(m), t\geq 0).\]
Since $L_p\cap L_q(m)$ is dense in $L_q(m)$ we can thus extend $T$ to a $C_0$-semigroup $T_q$ on $L_q(m)$. We will then say that \emph{$T$ also acts on $L_q(m)$}.
Let $A_q$ be the generator of $T_q$. Then $A_q u = A u$ for all $u\in L_p\cap L_q(m)$.

For formulating our main theorem, we shall need the following notion from semigroup theory, see e.g.~\cite[p 61]{EngelNagel}.
\begin{definition}
 Let $X$ be a Banach space, $T$ a $C_0$-semigroup on $X$ with generator $A$. Then $A^\sigma$ given by
 \[
    A^\sigma \coloneqq \left\{ (x',y')\in X'\times X'; \text{w}^*\text{-}\lim_{t\downarrow 0} \frac{1}{t}(T(t)'x'-x')=y'\right\}.
 \]
 is called the \emph{weak${}^*$-generator} of $T(\cdot)'$. 
\end{definition}

\begin{remark}\label{rem:Engel_Nagel_dual} It can be shown that $A^\sigma=A'$, the dual operator of the generator $A$ of $T$, see e.g.~\cite[p 61]{EngelNagel}. If $X$ is a Banach lattice, then so is $X'$ and if $(\lambda-A)^{-1}$ is positive for some real $\lambda\in \rho(A)$, then so is $(\lambda-A')^{-1}=\left((\lambda-A)^{-1}\right)'$. Note that, in general, $A'$ need not be densely defined anymore.
\end{remark}

\begin{theorem}\label{thm:main} Let $S,T$ be positive $C_0$-semigroups on $L_2(m)$ with generators $A_S, A_T$, respectively. Assume that $S$ also acts on $L_p(m)$ and $T$ also acts on $L_q(m)$. 
  Let $0\leq f\in L_p(m)$, $0\leq g'\in L_{p'}(m)$, $C\geq 0$.
  Then the following are equivalent:
  \begin{enumerate}
    \item
      For $t\geq 0$ and $0\leq u\in L_2\cap L_q(m)$, $0\leq v'\in L_2\cap L_{p'}(m)$ we have
      \[\dupa{T(t) u}{v'}_{2,2} \leq \dupa{S(t) u}{v'}_{2,2} + C\int_0^t \dupa{T(t-s)u}{g'}_{q,q'} \dupa{f}{S(s)'v'}_{p,p'}\, ds.\]
    \item
      For all $0\leq u\in D(A_T)\cap L_q(m)$, $0\leq v'\in D(A_S^{\sigma})\cap L_{p'}(m)$ we have
      \[\dupa{A_Tu}{v'}_{2,2} \leq \dupa{u}{A_S^{\sigma} v'}_{2,2} + C \dupa{u}{g'}_{q,q'}\dupa{f}{v'}_{p,p'}.\]
    \item
      For $\lambda\in\R$ sufficiently large and $0\leq u\in L_2\cap L_q(m)$, $0\leq v'\in L_2\cap L_{p'}(m)$ we have
      \[\dupa{(\lambda-A_T)^{-1} u}{v'}_{2,2} \leq \dupa{(\lambda-A_S)^{-1} u}{v'}_{2,2} + C\dupa{(\lambda-A_T)^{-1} u}{g'}_{q,q'} \dupa{f}{(\lambda-A_S^{\sigma})^{-1}v'}_{p,p'}.\]
  \end{enumerate}
\end{theorem}

\begin{proof}
  ``(a)$\Rightarrow$(b)'': Let $0\leq u\in D(A_T)\cap L_q(m)$, $0\leq v'\in D(A_S^{\sigma})\cap L_{p'}(m)$. For $t>0$ we have
  \begin{align*}
    & \dupa{\frac{1}{t}(T(t) u-u)}{v'}_{2,2} \\
    & = \dupa{\frac{1}{t} T(t) u}{v'}_{2,2} - \dupa{\frac{1}{t}u}{v'}_{2,2} \\
    & \leq \dupa{\frac{1}{t}(S(t) u-u)}{v'}_{2,2} + C \frac{1}{t}\int_0^t  \dupa{T(t-s)u}{g'}_{q,q'} \dupa{f}{S(s)'v'}_{p,p'}\, ds.
  \end{align*}
  Now, the limit $t\to 0$ yields the assertion by Lemma \ref{lem:conv_int}.
  
  ``(b)$\Rightarrow$(c)'': Let $\lambda\in \R$ be sufficiently large, $0\leq u\in L_2\cap L_q(m)$, $0\leq v'\in L_2\cap L_{p'}(m)$. 
  Then $0\leq (\lambda-A_T)^{-1}u=:\tilde{u}\in D(A_T)\cap L_q(m)$ and $0\leq (\lambda-A_S^\sigma)^{-1}v'=:\tilde{v}'\in D(A_S^\sigma)\cap L_{p'}(m)$.
  Furthermore,
  \begin{align*}
    & \dupa{(\lambda-A_T)^{-1} u}{v'}_{2,2} - \dupa{ u}{(\lambda-A_S^\sigma)^{-1}v'}_{2,2}\\ & = \dupa{A_T\tilde{u}}{\tilde{v}'}_{2,2} - \dupa{\tilde{u}}{A_S^\sigma \tilde{v}'}_{2,2} \\
    & \leq C\dupa{\tilde{u}}{g'}_{q,q'} \dupa{f}{\tilde{v}'}_{p,p'}\\
    & = C \dupa{(\lambda-A_T)^{-1}u}{g'}_{q,q'} \dupa{f}{(\lambda-A_S^{\sigma})^{-1}v'}_{p,p'}.
  \end{align*}
  
  ``(c)$\Rightarrow$(a)'': Let $0\leq u\in L_2\cap L_q(m)$, $0\leq v'\in L_2\cap L_{p'}(m)$. Let $\lambda\in \R$ be sufficiently large. 
  By induction on $n\in\N$ we obtain 
  \begin{align*} 
    & \dupa{(\lambda-A_T)^{-n} u}{v'}_{2,2} \\
    & \leq \dupa{(\lambda-A_S)^{-n} u}{v'}_{2,2} + C \sum_{l=1}^{n} \dupa{(\lambda-A_T)^{l-n-1}u}{g'}_{q,q'} \dupa{f}{(\lambda-A_S^\sigma)^{-l} v'}_{p,p'}.
  \end{align*}
  Indeed, the case $n=1$ holds by assumption. For the inductive step from $n\in \mathbb{N}$ to $n+1$ we observe
  \begin{align*}
   &\dupa{(\lambda-A_T)^{-n-1}u}{v'}_{2,2}\\&=\dupa{(\lambda-A_T)^{-1}(\lambda-A_T)^{-n}u}{v'}_{2,2} \\
    &\leq \dupa{(\lambda-A_S)^{-1}(\lambda-A_T)^{-n}u}{v'}_{2,2} \\
    &\quad + C\dupa{(\lambda-A_T)^{-1}(\lambda-A_T)^{-n}u}{g'}_{q,q'}\dupa{f}{(\lambda -A_S^\sigma)^{-1}v'}_{p,p'} \\
    &\leq \dupa{(\lambda-A_T)^{-n}u}{(\lambda-A_S^\sigma)^{-1}v'}_{2,2} \\
    &\quad + C\dupa{(\lambda-A_T)^{-1}(\lambda-A_T)^{-n}u}{g'}_{q,q'}\dupa{f}{(\lambda -A_S^\sigma)^{-1}v'}_{p,p'} \\
    &\leq \dupa{(\lambda-A_S)^{-n}u}{(\lambda-A_S^\sigma)^{-1}v'}_{2,2}
    \\ &\quad +  C \sum_{l=1}^{n} \dupa{(\lambda-A_T)^{l-n-1}u}{g'}_{q,q'} \dupa{f}{(\lambda-A_S^\sigma)^{-l}(\lambda-A_S^\sigma)^{-1} v'}_{p,p'}\\
    &\quad + C\dupa{(\lambda-A_T)^{-1}(\lambda-A_T)^{-n}u}{g'}_{q,q'}\dupa{f}{(\lambda -A_S^\sigma)^{-1}v'}_{p,p'}.
  \end{align*}
  
  Let $t>0$, and let $n\in\N$ be sufficiently large. Then
  \begin{align*}
  & \dupa{\left(1-\frac{tA_T}{n}\right)^{-n} u}{v'}_{2,2} \\
  & = \dupa{\left(\frac{n}{t}\right)^n\left(\frac{n}{t}-A_T\right)^{-n} u}{v'}_{2,2} \\
  & \leq \dupa{\left(\frac{n}{t}\right)^n\left(\frac{n}{t}-A_S\right)^{-n} u}{v'}_{2,2} 
  \\ &\quad + Ct \frac{1}{n} \sum_{l=1}^{n} \dupa{\left(\frac{n}{t}\right)^{n+1-l}\left(\frac{n}{t}-A_T\right)^{l-n-1}u}{g'}_{q,q'} \dupa{f}{\left(\frac{n}{t}\right)^{l}\left(\frac{n}{t}-A_S^\sigma\right)^{-l} v'}_{p,p'} \\
  & = \dupa{\left(1-\frac{tA_S}{n}\right)^{-n} u}{v'}_{2,2} 
  \\ &\quad + Ct \frac{1}{n} \sum_{l=1}^{n} \dupa{\left(\left(1-\frac{tA_T}{n}\right)^{-1}\right)^{n+1-l} u}{g'}_{q,q'} \dupa{f}{\left(\left(1-\frac{tA_S^\sigma}{n}\right)^{-1}\right)^l v'}_{p,p'}.
  \end{align*}
  Set $f_{n}(t,l):= \dupa{\left(\left(1-\frac{tA_T}{n}\right)^{-1}\right)^{l} u}{g'}_{q,q'}$, $g_{n}(t,l) := \dupa{f}{\left(\left(1-\frac{tA_S^\sigma}{n}\right)^{-1}\right)^l v'}_{p,p'}$.
  By Lemma \ref{lem:convolutions} and the exponential formula we obtain, as $n\to\infty$,
  \begin{align*}
    \dupa{T(t)u}{v'}_{2,2} & \leq \dupa{S(t) u}{v'}_{2,2} + C t\int_0^1 \dupa{T(t(1-s))u}{g'}_{q,q'} \dupa{f}{S(ts)'v'}_{p,p'}\, ds \\
    & = \dupa{S(t) u}{v'}_{2,2} + C \int_0^t \dupa{T(t-s)u}{g'}_{q,q'} \dupa{f}{S(s)'v'}_{p,p'}\, ds.
  \end{align*}\qed
\end{proof}

We can prove a similar statement to Theorem \ref{thm:main} (with essentially the same proof) by using the exponential growth bound for the semigroup $T$ on $L_q(m)$ and correspondingly the norm estimate for the resolvent $(\lambda-A_T)^{-1}$ on $L_q(m)$:

\begin{theorem}\label{thm:main_norm} Let $S,T$ be positive $C_0$-semigroups on $L_2(m)$ with generators $A_S, A_T$, respectively. Assume that $S$ also acts on $L_p(m)$ and $T$ also acts on $L_q(m)$. 
  Let $M\geq 1$, $\omega\in\R$ such that $\norm{T(t)u}_q\leq Me^{\omega t}\norm{u}_q$ for all $t\geq 0$ and $u\in L_q(m)$.
    
  Let $0\leq f\in L_p(m)$.
  Then the following are equivalent:
  \begin{enumerate}
    \item
      There exists $C\geq 0$ such that for all $t\geq 0$ and $0\leq u\in L_2\cap L_q(m)$, $0\leq v'\in L_2\cap L_{p'}(m)$ we have
      \[\dupa{T(t) u}{v'}_{2,2} \leq \dupa{S(t) u}{v'}_{2,2} + C\int_0^t e^{\omega (t-s)}\norm{u}_q \dupa{f}{S(s)'v'}_{p,p'}\, ds.\]
    \item
      There exists $C\geq 0$ such that for all $0\leq u\in D(A_T)\cap L_q(m)$, $0\leq v'\in D(A_S^{\sigma})\cap L_{p'}(m)$ we have
      \[\dupa{A_Tu}{v'}_{2,2} \leq \dupa{u}{A_S^{\sigma} v'}_{2,2} + C \norm{u}_q \dupa{f}{v'}_{p,p'}.\]
    \item
      There exists $C\geq 0$ such that for all $\lambda\in\R$ sufficiently large and $0\leq u\in L_2\cap L_q(m)$, $0\leq v'\in L_2\cap L_{p'}(m)$ we have
      \begin{align*}\dupa{(\lambda-A_T)^{-1} u}{v'}_{2,2} &\leq \dupa{(\lambda-A_S)^{-1} u}{v'}_{2,2} \\ &\quad + \frac{C}{\lambda-\omega}\norm{u}_q \dupa{f}{(\lambda-A_S^{\sigma})^{-1}v'}_{p,p'}.\end{align*}
  \end{enumerate}
\end{theorem}

Now, instead of a weak fomulation (using dual pairings) in Theorem \ref{thm:main_norm}, we want to prove a strong verion of this Theorem.
We need some preparation.

\begin{lemma}
\label{lem:positive_test_BL} Let $f\in L_p(m)$, $g\in L_q(m)$.
The following are equivalent:
  \begin{enumerate}
    \item $f+g\geq 0$.
    \item $\langle f+g,v'\rangle_{L_p(m)+L_q(m), L_{p'}\cap L_{q'}(m)} \geq 0$ for all $0\leq v'\in L_{p'}\cap L_{q'}(m)$.
  \end{enumerate}
\end{lemma}

As a Corollary, we can obtain the strong version of Theorem \ref{thm:main_norm}.

\begin{corollary}
\label{cor:strong_version} 
Let $S,T$ be positive $C_0$-semigroups on $L_2(m)$ with generators $A_S, A_T$, respectively. Assume that $S$ also acts on $L_p(m)$ and $T$ also acts on $L_q(m)$. Let $M\geq 1$, $\omega\in\R$ such that $\norm{T(t)u}_q\leq Me^{\omega t}\norm{u}_q$ for all $t\geq 0$ and $u\in L_q(m)$.
  
  Let $0\leq f\in L_p(m)$.
  Then the following are equivalent:
  \begin{enumerate}
    \item
      There exists $C\geq 0$ such that for all $t\geq 0$ and $0\leq u\in L_2\cap L_q(m)$, we have
      \[T(t) u \leq S(t) u +  C\int_0^t e^{\omega (t-s)}\norm{u}_q S(s)f\, ds.\]
    \item
      There exists $C\geq 0$ such that for all $0\leq u\in D(A_T)\cap L_q(m)$, $0\leq v'\in D(A_S^{\sigma})\cap L_{p'}(m)$ we have
      \[\dupa{A_Tu}{v'}_{2,2} \leq \dupa{u}{A_S^{\sigma} v'}_{2,2} + C \norm{u}_q\dupa{f}{v'}_{p,p'}.\]
    \item
      There exists $C\geq 0$ such that for all $\lambda\in\R$ sufficiently large and $0\leq u\in L_2\cap L_q(m)$ we have
      \[(\lambda-A_T)^{-1} u \leq (\lambda-A_S)^{-1} u + \frac{C}{\lambda-\omega}\norm{u}_q  (\lambda-A_S)^{-1}f.\]
  \end{enumerate}
\end{corollary}
\begin{proof} 
  Using Theorem \ref{thm:main_norm}, we realize that it suffices to observe that the statements (a),(b) and (c) from Theorem \ref{thm:main_norm} are equivalent to the corresponding ones here. For this, note that there is nothing to show for the statement (b) and that the remaining equivalences follow from Lemma \ref{lem:positive_test_BL}. \qed
\end{proof}

The main result of \cite{SeifertWingert2014} can be seen as a variant of Corollary \ref{cor:strong_version}.
There, the authors worked with $p=\infty$, $p'=1$ and $f=\1_\Omega$ (note that we require $p<\infty$ here; in order to obtain the case $p=\infty$ one needs another duality argument as in \cite{SeifertWingert2014}).

\section{Kernel estimates}\label{sec:kernel_estimate}

Let $(\Omega,\mathcal{F},m)$ be $\sigma$-finite.
Let $S,T$ be positive $C_0$-semigroups on $L_2(m)$ with generators $A_S, A_T$, respectively, $1\leq p,q<\infty$.
Assume $S$ has a kernel $k^S\from(0,\infty)\times\Omega^2\to \R$ on $L_2\cap L_p(m)$,  i.e.
\[S(t) u = \int_\Omega k^S(t,\cdot,y)u(y)\, dm(y) \quad(u\in L_2\cap L_p(m), t> 0),\]
and similarly let $T$ have a kernel $k^T$ on $L_2\cap L_q(m)$.

We can now use our result to prove kernels estimates for $T$ in terms of the kernel for $S$.

\begin{corollary}
\label{cor:kernel_estimate}
  In the above situation, let $S$ also act on $L_p(m)$ having a kernel $k^S$ on $L_2\cap L_p(m)$, and $T$ also act on $L_q(m)$ having a kernel $k^T$ on $L_2\cap L_q(m)$.
  Assume there exists $0\leq f\in L_p(m)$, $0\leq g'\in L_{q'}(m)$, $C\geq0$ such that 
  \[\dupa{A_Tu}{v'}_{2,2} \leq \dupa{u}{A_S^{\sigma} v'}_{2,2} + C \dupa{u}{g'}_{q,q'}\dupa{f}{v'}_{p,p'}\]
  for all $0\leq u\in D(A_T)\cap L_q(m)$, $0\leq v'\in D(A_S^{\sigma})\cap L_{p'}(m)$.
  Then
  \[k^T(t,x,y)\leq k^S(t,x,y) + C\int_0^t  \int_\Omega k^T(t-s,w,y)g'(w)\,dm(w) \int_\Omega  k^S(s,x,z)f(z)\,dm(z)\, ds\]
  for $m^2$-a.a.\ $(x,y)\in\Omega^2$ and $t> 0$.
\end{corollary}

\begin{proof}
  The kernel estimate follows from the corresponding estimate for the semigroups in Theorem \ref{thm:main} and reasoning as in the proof of \cite[Korollar 2.1.11]{w2011}.\qed
\end{proof}

\begin{remark}
  Note that we can choose $p=q=2$ in the preceeding corollary. 
\end{remark}

In case $q=1$, we can make use of Corollary \ref{cor:strong_version} to even deduce the existence of a kernel $k^T$ for $T$.

\begin{corollary}
\label{cor:kernel_estimate_norm}
  In the above situation with $q=1$, let $S$ also act on $L_p(m)$ having a kernel on $L_2\cap L_p(m)$, and $T$ also act on $L_1(m)$.
  Assume there exists $0\leq f\in L_p(m)$, $C_0\geq 0$ such that 
  \[\dupa{A_Tu}{v'}_{2,2} \leq \dupa{u}{A_S^{\sigma} v'}_{2,2} + C_0 \norm{u}_{1}\dupa{f}{v'}_{p,p'}\]
  for all $0\leq u\in D(A_T)\cap L_1(m)$, $0\leq v'\in D(A_S^{\sigma})\cap L_{p'}(m)$.

  Then $T$ has a kernel $k^T\from(0,\infty)\times\Omega^2\to \R$ on $L_2\cap L_q(m)$, and there exists $C\geq 0$ such that
  \[k^T(t,x,y)\leq k^S(t,x,y) + C\int_0^t e^{\omega (t-s)} \int_\Omega k^S(s,x,z)f(z)\,dm(z)\, ds\]
  for $m^2$-a.a.\ $(x,y)\in\Omega^2$ and $t> 0$.
\end{corollary}

\begin{proof}
  The existence of a kernel for $T$ follows from \cite[Theorem 5.9]{aa2002} (note that $\norm{u}_1 = \int 1\cdot u\,dm$ for $0\leq u\in L_2\cap L_1(m)$).
  The kernel estimate then follows from the corresponding estimate for the semigroups in Corollary \ref{cor:strong_version}(a) and reasoning as in the proof of \cite[Korollar 2.1.11]{w2011}.\qed
\end{proof}


\section{Application to perturbations of forms by jump parts}\label{sec:application_forms}

Let $(\Omega,\mathcal{F},m)$ be a measure space.
We give a short introduction to forms on $L_2$-spaces; for more information see e.g.\ \cite{Ouhabaz2005,isem18}.

A bilinear map $\form\from D(\form)\times D(\form)\to \R$, where $D(\form)$ is a subspace of $L_2(m)$, is called a \emph{form}. We write $\form(u):=\form(u,u)$ for the corresponding quadratic form.
A form $\form$ is \emph{densely defined} if $D(\form)$ is dense in $L_2(m)$. It is called \emph{accretive} if $\form(u,u)\geq 0$ for all $u\in D(\form)$.
$\form$ is called \emph{continuous} if there exists $M\geq 0$ such that
\[\abs{\form(u,v)}\leq M\norm{u}_\form \norm{v}_\form \quad(u,v\in D(\form)),\]
where $\norm{\cdot}_\form := \bigl(\form(\cdot) + \norm{\cdot}_2^2\bigr)^{1/2}$ is the \emph{form norm}. We say that $\form$ is closed if $(D(\form),\norm{\cdot}_\form)$ is complete.

\begin{remark}
  Let $\form$ be densely defined, accretive, continuous and closed. Then we can associate an operator $A$ to $\form$ via
  \[A \coloneqq \set{(u,v)\in L_2(m)\times L_2(m);\; u\in D(\form),\, \form(u,\varphi) = \sp{v}{\varphi} \quad(\varphi\in D(\form))}.\]
  Note that $A$ is an \rm{m}-accretive operator, i.e.\ $\sp{Au}{u}\geq 0$ for all $u\in D(A)$ and $R(I+A) = L_2(m)$. Furthermore, $-A$ generates a contractive $C_0$-semigroup $(e^{-tA})_{t\geq0}$ on $L_2(m)$. 
\end{remark}

The following remarks describing Ouhabaz type criteria for invariant subsets are consequences of \cite[Proposition 2.9]{ArendtterElst2012} and \cite[Remark 9.3]{isem18}; see also \cite[Theorem 2.6 and Theorem 2.13]{Ouhabaz2005} and \cite[Theorem 10.12]{isem18} for related results.

\begin{remark}
\label{rem:positive}
  Let $\form$ be a densely defined accretive continuous closed form in $L_2(m)$ (over $\R$), 
  $A$ the associated operator and $T = (e^{-tA})_{t\geq0}$ the associated $C_0$-semigroup. 
  The following are equivalent.
  \begin{enumerate}
    \item $T$ is positive, i.e.\ $T(t)u\geq 0$ for all $0\leq u\in L_2(m)$, $t\geq0$.
    \item For $u\in D(\form)$ we have $u^+\in D(\form)$ and $\form(u^+,u^-)\leq 0$.
  \end{enumerate}
\end{remark}

\begin{remark}
\label{rem:Linfcontractive}
  Let $\form$ be a densely defined accretive continuous closed form in $L_2(m)$ (over $\R$), 
  $A$ the associated operator and $T = (e^{-tA})_{t\geq0}$ the associated $C_0$-semigroup. 
  The following are equivalent.
  \begin{enumerate}
    \item $T$ is positive and $L_\infty$-contractive, i.e.\ $\norm{T(t)u}_\infty \leq \norm{u}_\infty$ for all $u\in L_2(m)\cap L_\infty(m)$, $t\geq0$.
    \item For $u\in D(\form)$ we have $u\wedge 1 \in D(\form)$ and $\form(u\wedge 1, (u-1)^+)\geq 0$.
  \end{enumerate}
\end{remark}

\begin{remark}
\label{rem:L1contractive}
  Let $\form$ be a densely defined accretive continuous closed form in $L_2(m)$ (over $\R$), 
  $A$ the associated operator and $T = (e^{-tA})_{t\geq0}$ the associated $C_0$-semigroup. 
  The following are equivalent.
  \begin{enumerate}
    \item $T$ is positive and $L_1$-contractive, i.e.\ $\norm{T(t)u}_1 \leq \norm{u}_1$ for all $u\in L_2(m)\cap L_1(m)$, $t\geq0$.
    \item For $u\in D(\form)$ we have $u\wedge 1 \in D(\form)$ and $\form((u-1)^+,u\wedge 1)\geq 0$.
  \end{enumerate}
\end{remark}

Let $\form_0$ be a densely defined accretive continuous closed form on $L_2(m)$ with associated operator $A_0$ and $C_0$-semigroup $S = (e^{-tA_0})_{t\geq 0}$, such that
$S$ is positive and contractive in $L_1(m)$ or $L_\infty(m)$ (cf.\ the Remarks \ref{rem:positive}, \ref{rem:Linfcontractive} and \ref{rem:L1contractive}).
Then $S$ also acts on $L_p(m)$ for all $1\leq p\leq 2$ or $2\leq p<\infty$, respectively, by interpolation.
Note that if $\form_0$ is \emph{symmetric}, i.e.\ $\form_0(u,v) = \form_0(v,u)$ for all $u,v\in D(\form_0)$, then $A_0$ is self-adjoint and also $S(t)$ is self-adjoint for all $t\geq0$.
By duality, we then obtain that $S$ acts on the whole scale of $L_p(m)$-spaces.

Furthermore, let us assume 
\[D_{\fin}:=\set{u\in D(\form_0);\; m([u\neq 0])<\infty}\] is dense in $D(\form_0)$.

Let $j\from \Omega\times\Omega\to \R$ be measurable, $j\geq 0$, such that $\int_B j(x,y)\,dm^2(x,y)<\infty$ for all Borel sets $B\subseteq \Omega^2$ such that $m^2(B)<\infty$.
Consider
\begin{align*}
  D(\form) & := \set{u\in D(\form_0);\; \int_{\Omega\times\Omega} (u(x)-u(y))^2 j(x,y)\,dm^2(x,y) < \infty}, \\
  \form(u,v) & := \form_0(u,v) + \int_{\Omega\times\Omega} (u(x)-u(y))(v(x)-v(y)) j(x,y)\,dm^2(x,y).
\end{align*}

\begin{lemma}
  $\form$ is densely defined, accretive, continuous and closed.
\end{lemma}

\begin{proof}
  To show that $\form$ is densely defined it suffices to approximate elements of $D_{\fin}$ by elements of $D(\form)$. Let $u\in D_{\fin}$. Since $S$ is positive without loss of generality we may assume that $u\geq 0$.
  Since $S$ is contractive in $L_1(m)$ or in $L_\infty(m)$, we have $u_n:=u\wedge n\in D(\form_0)$ for all $n\in\N$. Since $[ u_n\neq 0] =[ u\neq 0]$ for all $n\in\N$ we observe
  \[\int_{\Omega\times\Omega} (u_n(x)-u_n(y))^2 j(x,y) \, dm^2(x,y) \leq (2n)^2 \int_{[u\neq 0] \times [u\neq 0]} j(x,y)\, dm^2(x,y) <\infty \quad(n\in\N),\]
  i.e.\ $u_n\in D(\form)$ for all $n\in\N$. Since $u_n\to u$ in $L_2(m)$ the form $\form$ is densely defined.
  
  Since $j\geq 0$ we have $\form(u,u)\geq \form_0(u,u)\geq 0$ for all $u\in D(\form)$, implying that $\form$ is accretive.
  
  Since $\form_0$ is continuous there exists $M\geq 0$ such that $\abs{\form_0(u,v)}\leq M \norm{u}_{\form_0}\norm{v}_{\form_0}$ for all $u,v\in D(\form_0)$.
  Since the bilinear form $(u,v)\mapsto \int_{\Omega\times\Omega} (u(x)-u(y))(v(x)-v(y)) j(x,y)\,dm(x,y)$ is symmetric and accretive we obtain $\norm{\cdot}_{\form_0} \leq \norm{\cdot}_\form$ and by the Cauchy-Schwarz inequality
  \begin{align*}
    & \int_{\Omega\times\Omega} (u(x)-u(y))(v(x)-v(y)) j(x,y)\,dm^2(x,y) \\
    & \leq \left(\int_{\Omega\times\Omega} (u(x)-u(y))^2 j(x,y)\,dm^2(x,y)\right)^{1/2} \left(\int_{\Omega\times\Omega} (v(x)-v(y))^2 j(x,y)\,dm^2(x,y)\right)^{1/2} \\
    & \leq \norm{u}_\form \norm{v}_\form \quad(u,v\in D(\form)).
  \end{align*}
  Thus,
  \[\abs{\form(u,v)} \leq M\norm{u}_\form \norm{v}_\form + \norm{u}_\form \norm{v}_\form = (M+1)\norm{u}_\form \norm{v}_\form \quad(u,v\in D(\form)),\]
  i.e.\ $\form$ is continuous.
  
  To show closedness of $\form$ let $(u_n)_n$ be a $\norm{\cdot}_\form$-Cauchy sequence in $D(\form)$ such that $u_n\to u$ in $L_2(m)$. Without loss of generality (by choosing a suitable subsequence) we may assume that $u_n\to u$ $m$-a.e.
  Since $ \form_0\leq  \form$ the sequence $(u_n)_n$ is also a $\norm{\cdot}_{\form_0}$-Cauchy sequence. Since $\form_0$ is closed we obtain $u\in D(\form_0)$ and $\norm{u_n-u}_{\form_0}\to 0$.
  By Fatou's lemma we have
  \begin{align*}
    & \int_{\Omega\times\Omega} \bigl((u-u_n)(x) - (u-u_n)(y)\bigr)^2j(x,y)\,dm^2(x,y) \\
    & = \int_{\Omega\times\Omega} \liminf_{m\to\infty} \bigl((u_m-u_n)(x) - (u_m-u_n)(y)\bigr)^2j(x,y)\,dm^2(x,y) \\
    & \leq \liminf_{m\to\infty} \int_{\Omega\times\Omega} \bigl((u_m-u_n)(x) - (u_m-u_n(y)\bigr)^2j(x,y)\,dm^2(x,y) \\
    & \leq \liminf_{m\to\infty} \norm{u_m-u_n}_{\form}^2\to 0 \quad(n\to\infty).
  \end{align*}
  In particular, $u-u_n\in D(\form)$ and therefore $u\in D(\form)$. Furthermore,
  \[ \form(u-u_n) \leq  \form_0(u-u_n) + \liminf_{m\to\infty} \int_{\Omega\times\Omega} \bigl((u_m-u_n)(x) - (u_m-u_n(y)\bigr)^2j(x,y)\,dm^2(x,y)\to 0\]
  which implies $\norm{u-u_n}_\form\to 0$, i.e.\ $\form$ is closed. \qed
\end{proof}

Let $A$ be the operator associated with $\form$ and $T = (e^{-tA})_{t\geq 0}$ be the $C_0$-semigroup.

\begin{lemma}
  The semigroup $T$ is positive.
  If $S$ is contractive in $L_1(m)$, then $T$ is contractive in $L_1(m)$, and therefore acts on $L_q(m)$ for all $1\leq q\leq2$.
  If $S$ is contractive in $L_\infty(m)$, then $T$ is contractive in $L_\infty(m)$, and therefore acts on $L_q(m)$ for all $2\leq q<\infty$.
\end{lemma}

\begin{proof}  
  Let $u\in D(\form)$. Since $S$ is positive we have $u^+\in D(\form_0)$ and $\form_0(u^+,u^-)\leq 0$. 
  Note that $(u^+(x)-u^+(y))^2\leq (u(x)-u(y))^2$ for $m^2$-a.a.\ $(x,y)\in\Omega^2$. This implies $u^+\in D(\form)$ and 
  \[\form(u^+,u^-) = \form_0(u^+,u^-) + \int_{\Omega\times\Omega} (u^+(x)-u^+(y))(u^-(x)-u^-(y)) j(x,y)\,dm^2(x,y) \leq 0.\]
  Hence, $T$ is positive.
 
  Assume that $S$ is contractive in $L_1(m)$. Then we have $u\wedge 1 \in D(\form_0)$ and $\form_0((u-1)^+, u\wedge 1)\geq 0$.
  Let $u\in D(\form)$. Since 
  \[\Omega\times\Omega = \bigl([u\geq1] \times [u\geq 1]\bigr) \cup \bigl([u\geq1]\times [u<1]\bigr)\cup \bigl([u<1]\times[u\geq1]\bigr) \cup \bigl([u<1]\times[u<1]\bigr),\]
  an easy computation yields
  \begin{align*}
    & \form((u-1)^+, u\wedge 1) \\
    & = \form_0((u-1)^+, u\wedge 1)
    \quad + \int_{\Omega\times\Omega} \bigl((u-1)^+(x)-(u-1)^+(y)\bigr)\bigl((u\wedge 1)(x) - (u\wedge 1)(y)\bigr) j(x,y)\,dm^2(x,y) \\
    & \geq 0.
  \end{align*}
  Hence, $T$ is $L_1$-contractive. Since $T$ is $L_1$-contractive and also $L_2$-contractive, it acts on the whole scale of $L_q(m)$ with $1\leq q\leq 2$ by interpolation.

  Now, assume that $S$ is contractive in $L_\infty(m)$. Then we have $u\wedge 1 \in D(\form_0)$ and $\form_0(u\wedge 1, (u-1)^+)\geq 0$.
  As in the case of $S$ being $L_1$-contractive, we obtain
  \begin{align*}
    \form(u\wedge 1, (u-1)^+) & \geq 0 \quad(u\in D(\form)).
  \end{align*}
  Hence, $T$ is $L_\infty$-contractive. Since $T$ is $L_\infty$-contractive and also $L_2$-contractive, it acts on the whole scale of $L_q(m)$ with $2\leq q<\infty$ by interpolation. \qed
\end{proof}
Note that if $\form_0$ is symmetric then also $\form$ is symmetric and hence $A$ is self-adjoint and $T(t)$ is self-adjoint for all $t\geq 0$. 

We will now focus on the case $q=1$ and $q' = \infty$, meaning that $T$ is also acting on $L_1(m)$.
We distinguish the following cases:

\begin{case}[$1\leq p<\infty$]
\label{case:1}
  Assume
  \[\int_\Omega j(x,\cdot)\,dm(x), \int j(\cdot,y)\,dm(y)\in L_\infty(m).\]
\end{case}

\begin{case}[$1\leq p\leq 2$]
\label{case:2}
  Assume
  \[\int_\Omega j(x,\cdot)\,dm(x), \int j(\cdot,y)\,dm(y)\in L_\infty(m) + L_{\frac{2p}{2-p}}(m).\]
\end{case}
%
%

In either of these cases, we may further assume that $[x\mapsto \norm{j(x,\cdot)}_{\infty}], [y\mapsto \norm{j(\cdot,y)}_{\infty}]\in L_{p}(m)$.
For $0\leq u\in D(A)\cap L_q(m)$, $0\leq v\in D(A_0^\sigma)\cap L_{p'}(m)$, we compute using $j\geq 0$
\begin{align*}
  \sp{-Au}{v} - \sp{u}{-A_0^\sigma v} & = \form_0(u,v) - \form(u,v) \\
  & = -\int_{\Omega\times\Omega} (u(x)-u(y))(v(x)-v(y)) j(x,y)\,dm^2(x,y) \\
  & \leq \int_{\Omega\times\Omega} u(y)v(x)j(x,y)\, dm^2(x,y) + \int_{\Omega\times\Omega} u(x)v(y)j(x,y)\, dm^2(x,y) \\
  & = \int_\Omega v(x) \int_\Omega u(y) j(x,y)\,dm(y)\,dm(x)\\ & \quad  + \int_\Omega v(y) \int_\Omega u(x) j(x,y)\, dm(x)\, dm(y) \\
  & \leq \int_\Omega v(x) \norm{u}_{1}\norm{j(x,\cdot)}_{\infty}\,dm(x) + \int_\Omega v(y) \norm{u}_{1}\norm{j(\cdot,y)}_{\infty}\,dm(y) \\
  & = \norm{u}_{1} \dupa{[x\mapsto \norm{j(x,\cdot)}_{\infty}] + [y\mapsto \norm{j(\cdot,y)}_{\infty}]}{v}_{p,p'}.
\end{align*}

Thus, we can apply Corollary \ref{cor:strong_version} to obtain an estimate of the semigroup $T$ in terms of the semigroup $S$.
For the kernel estimates we apply Corollary \ref{cor:kernel_estimate_norm}.

\begin{theorem}
\label{thm:application_forms}
Let $(\Omega,\mathcal{F},m)$ be a $\sigma$-finite measure space. Let $\form_0, S, A_S, j, \form, T, A_T$ as above, where we assume one of the Cases \ref{case:1} or \ref{case:2} to be satisfied. Assume
$[x\mapsto \norm{j(x,\cdot)}_{\infty}], [y\mapsto \norm{j(\cdot,y)}_{\infty}]\in L_{p}(m)$.
Assume that $S$ has a kernel $k^S$ on $L_2\cap L_p(m)$. Then $T$ has a kernel $k^T$ on $L_2\cap L_1(m)$ satisfying
\[k^T(t,x,y)\leq k^S(t,x,y) + C\int_0^t \int_\Omega k^S(s,x,z)\bigl(\norm{j(z,\cdot)}_{\infty} + \norm{j(\cdot,z)}_{\infty}\bigr)\,dm(z)\,ds\]
for $m^2$-a.a.\ $(x,y)\in\Omega^2$, $t> 0$ and some $C\geq0$.
\end{theorem}

\begin{remark}
  In the situation of Theorem \ref{thm:application_forms} one can even show that $C=1$. In order to prove this one needs to realize that $T$ is a contraction on $L_1(m)$, that the corresponding $C_0$ from Corollary \ref{cor:kernel_estimate_norm} equals $1$ and that
  $C$ can be computed from $C_0$ and the norm estimate of the resolvent of $A_T$ on $L_1(m)$.
\end{remark}

%
%
%
%
%
%
%
%
%

\end{document}